\documentclass[reqno]{amsart}
\usepackage{amsmath}
\usepackage{amsfonts}
\usepackage{amsthm}
\usepackage{mathrsfs}
\usepackage{amssymb}

\newtheorem{theorem}{Theorem}
\newtheorem{lemma}{Lemma}

\theoremstyle{definition}
\theoremstyle{acknowledgement}

\newtheorem*{acknowledgement}{Acknowledgement}
\newtheorem{definition}[theorem]{Definition}

\theoremstyle{remark}

\usepackage{graphicx}

\newcommand{\R}{\mathbb{R}}

\newcommand{\N}{\mathbb{N}}
\newcommand{\Z}{\mathbb{Z}}

\newcommand{\cA}{\mathcal{A}}
\newcommand{\cB}{\mathcal{B}}
\newcommand{\cC}{\mathcal{C}}

\newcommand{\cF}{\mathcal{F}}
\newcommand{\cG}{\mathcal{G}}
\newcommand{\cH}{\mathcal{H}}

\newcommand{\Corr}{\mbox{Corr}}
\newcommand{\dist}{\mbox{dist}}

\numberwithin{equation}{section}
\numberwithin{theorem}{section}
\numberwithin{lemma}{section}
\numberwithin{remark}{section}
\begin{document}

\title[On Mixing Properties of Some INAR Models]{On Mixing Properties of Some INAR Models}

\dedicatory{Dedicated to the memory of Mikhail Gordin}
\author[Richard C.\ Bradley]{
Richard C.\ Bradley}
\address{Department of Mathematics\\
Indiana University\\
Bloomington, Indiana 47405, USA}

\email{bradleyr@indiana.edu}

\begin{abstract}  Strictly stationary INAR(1) (``integer-valued autoregressive processes of order 1'') with Poisson innovations are ``interlaced $\rho$-mixing''.
\end{abstract}

\maketitle

\section{Introduction}
\label{sc1}

   The INAR processes, or ``integer-valued autoregressive processes'', are a variant of the usual autoregressive processes in time series analysis. In various references, the INAR processes and variations on them have been studied as models to use in the statistical analysis of ``count data''.  See e.g.\  \cite{8}, \cite{12}, \cite{14}, \cite{17}, \cite{18}, and the references therein.  In \cite{17}, for certain INAR(1) processes (``integer-valued autoregressive processes of order~1''), and certain variations on them, strong mixing and even absolute regularity were verified, with exponential mixing rate.

   In the study of INAR processes, the $\rho$-mixing condition does not seem to have gotten much attention, but it could perhaps play a useful role as well, given the extensive literature on limit theory under $\rho$-mixing that has been developed since early results on that topic such as in \cite{10} and \cite{16}.

   This note here will go in a little different direction. Within the INAR processes, one particularly prominent subclass is the strictly stationary INAR(1) processes with ``Poisson innovations''.  For that subclass (and some other related processes), absolute regularity with exponential mixing rate was already verified in \cite{17}. In this note, for that prominent subclass, we shall verify the $\rho^*$-mixing (``interlaced $\rho$-mixing'') condition, which is stronger than $\rho$-mixing. (Both of those latter two mixing conditions and that subclass of processes will be explicitly formulated below.) The Poisson innovations seem to facilitate the study of the $\rho^*$-mixing condition for that subclass. The techniques in this note involving $\rho^*$-mixing can apparently be extended to some limited extent to some other INAR processes, and even to some variations on them such as ones in \cite{17}. For simplicity, this note will be confined to just the subclass identified above.

   The processes in the prominent subclass discussed above are strictly stationary, countable-state Markov chains.  It is well known and elementary that for Markov chains, for either the $\rho$-mixing condition or the $\rho^*$-mixing condition, the mixing rate is automatically (at least) exponential. Now strictly stationary, finite-state, irreducible, aperiodic Markov chains are $\rho^*$-mixing (see \cite{1} or \cite[Theorem 7.15]{3}).  However, for strictly stationary, countable-state Markov chains in general, $\rho$-mixing does not imply $\rho^*$-mixing. (Counterexamples are constructed in \cite{2} and \cite{4}, with the ones in the latter reference being reversible.) For the INAR processes in general, and in particular for the (Markovian) INAR(1) processes whose innovations are not Poisson, there is more to explore regarding the $\rho$-mixing and $\rho^*$-mixing conditions and the connections between them.

   Now let us formulate the $\rho$-mixing and $\rho^*$-mixing conditions, define the subclass of processes that will be studied here, and then give the main result.

   Suppose $X:=(X_k$, $k\in\Z)$ is a strictly stationary sequence of random variables on a probability space $(\Omega,\cF,P)$. For any two $\sigma$-fields $\cA$ and $\cB \subset \cF$, define the ``maximal correlation'' \cite{9}:
   \[
   \rho(\cA,\cB) := \sup |\Corr(f,g)|
   \]
where the supremum is taken over all pairs of square-integrable random variables $f$ and $g$ such that $f$ is $\cA$-measurable and $g$ is $\cB$-measurable. For each positive integer $n$, define the following two dependence coefficients (for the given strictly stationary sequence $X$):
\begin{equation}\label{eq1.1}
\rho(X,n) := \rho(\sigma(X_k, k\le 0), \sigma(X_k,k\ge n))
\end{equation}
and
\begin{equation}\label{eq1.2}
\rho^*(X,n) := \sup \rho(\sigma(X_k,k\in S), \sigma(X_k, k\in T))
\end{equation}
where the supremum is taken over all pairs of nonempty, disjoint sets $S,T\subset \Z$ such that
\begin{equation}
\label{eq1.3}
\dist(S,T) := \min_{s\in S,t\in T} |s-t| \ge n.
\end{equation}
In (\ref{eq1.1}), (\ref{eq1.2}), and below, the notation $\sigma(\dots)$ means the $\sigma$-field generated by $(\dots)$. In (\ref{eq1.2})--(\ref{eq1.3}), the sets $S$ and $T$ can be ``interlaced'', with each set containing elements between ones in the other set.  The (strictly stationary) sequence $X$ is said to be ``$\rho$-mixing'' (a condition introduced in \cite{11}) if $\rho(X,n) \to 0$ as $n\to \infty$, and $\rho^*$-mixing (a condition apparently first studied in \cite{19}) if $\rho^*(X,n) \to0$ as $n\to \infty$. Obviously $\rho(X,n) \le \rho^*(X,n)$ for each $n\ge 1$, and (hence) $\rho^*$-mixing implies $\rho$-mixing.

The following terminology will be useful.

\begin{definition}\label{df1.1}
An ordered triplet $(\cA,\cB,\cC)$ of $\sigma$-fields $(\subset \cF)$ will be called a ``Markov triplet'' if any (hence all) of the following three equivalent conditions holds:

\renewcommand{\labelenumi}{(\roman{enumi})}

\begin{enumerate}
\item for all $C\in \cC$, $P(C|\cA \vee \cB) = P(C|\cB)$ a.s.;
\item for all $A\in \cA$ and all $C\in \cC$, $P(A\cap C|\cB) = P(A|\cB) \cdot P(C|\cB)$ a.s.;
\item for all $A\in \cA$, $P(A\mid \cB\vee \cC) = P(A\mid \cB)$ a.s.
\end{enumerate}

\end{definition}

The following elementary observation will be useful later on: If $(\cA,\cB,\cC)$ is a Markov triplet, then $(\cA \vee \cB, \cB,\cB\vee \cC)$ is a Markov triplet, and (hence) for any $\sigma$-fields $\cG \subset \cA\vee \cB$ and $\cH \subset \cB \vee \cC$, $(\cG, \cB, \cH)$ is a Markov triplet.

In what follows, $\N$ denotes the set of all positive integers, and $\overline{\N} := \N \cup \{0\}$ denotes the set of all nonnegative integers.

\begin{definition}\label{df1.2}
Suppose $a\in (0,1)$ and $\lambda>0$. A strictly stationary ``INAR(1) process with Poisson innovations'' (with parameters $a$ and $\lambda$), is a strictly stationary Markov chain $X:= (X_k, k\in \Z)$ with state space $\overline{\N}$, with $X$ having the following ``structural'' properties: There exist random variables $U_k, V_k$, $k\in \Z$ for which the following conditions hold:
\renewcommand{\labelenumi}{(\roman{enumi})}

\begin{enumerate}
\item For each $k\in \Z$, $X_k = U_k + V_k$.
\item For each $k\in \Z$ and each $x\in \overline{\N}$, the conditional distribution of $U_k$ given $\{X_{k-1}=x\}$ is binomial with parameters $x$ and $a$.
\item For each $k\in \Z$, the ordered triplet of $\sigma$-fields
\[
(\sigma(U_j, V_j, X_j, j\le k-1),\, \sigma(X_{k-1}), \sigma(U_k))
\]
is a Markov triplet.
\item For each $k\in \Z$, the distribution of the random variable $V_k$ is Poisson with mean~$\lambda$.
\item For each $k\in \Z$, the random variable $V_k$ is independent of the $\sigma$-field $\sigma(U_j,V_j,X_j,$ $j\le k-1)\vee \sigma(U_k)$.
\end{enumerate}
\end{definition}

In Definition \ref{df1.2}, for a given $k\in \Z$, the random variable $V_k$ is the  ``Poisson innovation''.  It is well known and elementary (see e.g.\ \cite{17}) that in the context of Definition \ref{df1.2}, the (invariant) marginal distribution of each $X_k$ is Poisson with mean $\lambda/(1-a)$.

Here is the main result of this note:

\begin{theorem}\label{th1.3}
Suppose $a\in (0,1)$ and $\lambda>0$; and suppose $X := (X_k, k\in \Z)$ is the strictly stationary INAR(1) process (Markov chain) in Definition~\ref{df1.2}, meeting all conditions there (including the Poisson ($\lambda$) ``innovations''). Then $X$ is $\rho^*$-mixing (with $\rho^*(X,n)\to 0$ at least exponentially fast as $n\to\infty$).
\end{theorem}

The proof of this theorem will be carried out through Sections \ref{sc2}, \ref{sc3}, and \ref{sc4} below. From that proof, one can see that the (of course exponential) mixing rate for $\rho^*$-mixing in Theorem \ref{th1.3} essentially depends only on (an upper bound for) the parameter $a$, not on~$\lambda$.

\section{Preliminaries}\label{sc2}
Throughout the rest of this note, the setting will be a probability space $(\Omega, \cF,P)$, rich enough to accommodate all random variables specified. Random variables are real-valued (and often integer-valued or even $\{0,1\}$-valued) unless specified otherwise.

Section \ref{sc2} here will be devoted to some lemmas that will be used in the proof of Theorem~\ref{th1.3}.

The following lemma is due to Cs\'aki and Fisher \cite{7}. (The proof given there has a flaw.  For a fully correct proof, see \cite{20} or \cite[Theorem 6.1]{3}.)

\begin{lemma}\label{lm2.1}
Suppose $\cA_n$ and $\cB_n$, $n\in \N$ are $\sigma$-fields $(\subset \cF)$, and the $\sigma$-fields $\cA_n\vee\cB_n$, $n\in\N$ are independent. Then
\[
\rho\left(\bigvee_{n\in\N} \cA_n, \bigvee_{n\in\N}\cB_n\right) = \sup_{n\in \N} \rho(\cA_n,\cB_n).
\]
\end{lemma}

Next, for any two $\sigma$-fields $\cA$ and $\cB$ $(\subset \cF)$, define the following measure of dependence:
\begin{equation}
\label{eq2.1}
\lambda(\cA,\cB) := \sup \frac{ |P(A\cap B) - P(A)P(B)|}{[P(A)]^{1/2}[P(B)]^{1/2}}
\end{equation}
where the supremum is taken over all pairs of events $A\in \cA$ and $B\in \cB$ such that $P(A)>0$ and $P(B)>0$.

\begin{lemma}
\label{lm2.2}
For any $\varepsilon >0$, there exists $\delta = \delta(\varepsilon)>0$ such that the following holds: If $\cA$ and $\cB$ are $\sigma$-fields such that $\lambda(\cA,\cB) \le \delta$, then $\rho(\cA,\cB) \le \varepsilon$.
\end{lemma}

Quite sharp versions of Lemma \ref{lm2.2} can be found in \cite{5}, \cite{6}, \cite[Theorem 4.15]{3}, and in a very sharp form, \cite{15}.

\begin{lemma}\label{lm2.3}
Suppose $0<\varepsilon\le 1/9$. Suppose $(X_1,X_2,X_3,\dots)$ is a sequence of random variables such that for each $n\ge 2$, $P(X_n =0\mid X_{n-1} =0) =1$ and
\[
P(X_n = 0 \mid\sigma(X_1,X_2,\dots,X_{n-1})) \ge 1-\varepsilon\,\,\mbox{a.s.}
\]
Then
\[
\lambda (\sigma(X_1,X_3,X_5,X_7,\dots), \sigma(X_2,X_4,X_6,X_8,\dots)) \le 3\varepsilon^{1/2}.
\]
\end{lemma}

A proof of Lemma \ref{lm2.3} can be found in \cite[Lemma 3.1]{1} or \cite[Theorem 5.21]{3}.  (In Lemma \ref{lm2.3}, the labeling of the ``absorbing state'' as $0$ is just for convenience.)

Next, for any (not necessarily stationary) sequence $X:= (X_k, k\in \Z)$ or $X := (X_k, k\in \overline{\N})$, define the dependence coefficients $\rho^*(n)$, $n\in\N$ by (\ref{eq1.2})--(\ref{eq1.3}). (In the case of index set $\overline{\N}$, the sets $S$ and $T$ are restricted to that set.)

In what follows, if $S$ is a nonempty finite set $\subset \overline{\N}$, $J$ is its cardinality, $X_k$, $k\in S$ are random variables, and (say) $f:\R^J \to \R$ is a Borel function, then the notation $f(X_k, k\in S)$ means $f(X_{k(1)},X_{k(2)},\dots,X_{k(J)})$ where $k(1)<k(2)<\dots<k(J)$ are the elements of $S$ in strictly increasing order.

\begin{lemma}
\label{lm2.4}
For any $a\in (0,1)$ and any $\varepsilon>0$, there exists a positive integer $m=m(a,\varepsilon)$ such that the following holds:

Suppose $\zeta_0$ is a $\{0,1\}$-valued random variable. Suppose $\eta := (\eta_1,\eta_2,\eta_3,\dots)$ is a sequence of independent, identically distributed $\{0,1\}$-valued random variables such that $P(\eta_1=1) = a$, with this sequence $\eta$ being independent of $\zeta_0$. For each $k\in \N$, define the $\{0,1\}$-valued random variable
\begin{equation}
\label{eq2.2}
\zeta_k := \zeta_0 \cdot \prod^k_{i=1} \eta_i.
\end{equation}
Then the random sequence $\zeta := (\zeta_0,\zeta_1,\zeta_2,\dots)$ satisfies
\begin{equation}
\label{eq2.3}
\rho^* (\zeta,m) \le \varepsilon.
\end{equation}
\end{lemma}

\begin{proof}
Suppose $a\in (0,1)$ and $\varepsilon >0$. Our first task is to define the positive integer $m = m(a,\varepsilon)$.

Referring to (\ref{eq2.1}), let $\delta = \delta(\varepsilon)>0$ be as in Lemma \ref{lm2.2}. Let $\gamma \in (0,1/9]$ be such that
\begin{equation}
\label{eq2.4}
3 \gamma^{1/2} \le \delta.
\end{equation}
Note that $\delta$ and (hence) $\gamma$ depend only on $\varepsilon$. Let $m = m(a,\varepsilon)$ be a positive integer such that
\begin{equation}
\label{eq2.5}
a^m \le \gamma.
\end{equation}
That completes the definition of $m=m(a,\varepsilon)$.

Now suppose the random variable $\zeta_0$, the random sequence $\eta$, and (then) the random sequence $\zeta$ are as in the statement of Lemma \ref{lm2.4}. Our task is to prove (\ref{eq2.3}).

Suppose $S$ and $T$ are any two nonempty, disjoint subsets of $\overline{\N}$ such that $\dist(S,T)$ $\ge m$. To complete the proof of (\ref{eq2.3}), it suffices to show that
\begin{equation}
\label{eq2.6}
\rho(\sigma(\zeta_k,k\in S), \sigma(\zeta_k, k\in T)) \le \varepsilon.
\end{equation}
By a standard measure-theoretic argument, it suffices to show (\ref{eq2.6}) in the case where both index sets $S$ and $T$ are finite. We make that assumption.

Just for convenience, without loss of generality (after switching $S$ and $T$ if necessary, and after enlarging $T$ by one element if necessary), we assume that the least and greatest elements of the set $S\cup T$ belong to $S$ and $T$ respectively. Then there exists a positive even  integer $L$ and nonempty, (pairwise) disjoint sets $Q_1,Q_2,\dots,Q_L \subset \overline{\N}$ with the following properties:
\begin{eqnarray}
\label{eq2.7}
&&S= \bigcup_{i\in \{1,3,5,\dots,L-1\}} Q_i; \nonumber\\
&&T = \bigcup_{i\in\{2,4,6,\dots,L\}} Q_i;\,\,\mbox{and} \nonumber\\
&&\forall\,\, i\in \{1,2,\dots,L-1\},\,\,\,m+[\max Q_i]\le [\min Q_{i+1}].
\end{eqnarray}

For each positive integer $J$, let $\phi_J: \{0,1\}^J \to \overline \N$ be a one-to-one function such that $\phi_J(0,0,\dots,0) =0$. For each $i\in \{1,2,\dots,L\}$, define the ($\N$-valued) random variable $X_i = \phi_{J(i)}(\zeta_k,k\in Q_i)$ where $J(i)$ is the cardinality of $Q_i$. Then
\begin{eqnarray}
\label{eq2.8}
&&\forall\,\, i\in \{1,2,\dots,L\}, \nonumber\\
&&\qquad  \sigma(X_i) = \sigma(\zeta_k, k\in Q_i)\,\,\,\mbox{and} \,\,\,\{X_i =0\} = \{\zeta_k =0 \,\, \forall\, k\in Q_i\};
\end{eqnarray}
and (hence)
\begin{eqnarray}
\label{eq2.9}
\sigma(\zeta_k,k\in S) &=&\sigma(X_1,X_3,X_5,\dots,X_{L-1})\quad\mbox{and}\nonumber\\
\sigma(\zeta_k,k\in T) &=& \sigma(X_2,X_4,X_6,\dots,X_L).
\end{eqnarray}

For each $k\in \N$, by (\ref{eq2.2}) and the assumptions in Lemma \ref{lm2.4}, one has that (i)~$\zeta_k = \zeta_{k-1} \cdot\eta_k$ and hence $\{\zeta_{k-1} =0\}\subset \{\zeta_k=0\}$, and (ii)~the $\sigma$-fields $\sigma(\eta_i, i\ge k)$ and $\sigma(\zeta_i, i\le k-1)$ are independent. These facts have the following two consequences:

First, by (\ref{eq2.7}) and (\ref{eq2.8}), for each $i\in \{2,3,\dots,L\}$, $\{X_{i-1} =0\} \subset \{X_i =0\}$ and hence $P(X_i =0 \mid X_{i-1} =0) =1$.

Second, for each $i\in \{2,3,\dots,L\}$, letting $j:= \max Q_{i-1}$, one has by (\ref{eq2.2}), (\ref{eq2.7}), and (\ref{eq2.8}) that $\{X_i =0\} \supset \bigcup^{j+m}_{u=j+1} \{\eta_u =0\}$, this latter event is independent of $\sigma(\zeta_k, k\le j)$ and hence independent of $\sigma(X_1,X_2,\dots,X_{i-1})$, and hence now by (\ref{eq2.5}), almost surely
\begin{eqnarray*}
P\left(X_i =0\mid \sigma(X_1,X_2,\dots,X_{i-1})\right) &\ge& P\left( \bigcup^{j+m}_{u=j+1}\{\eta_u =0\}\,\,\bigg|\,\, \sigma(X_1,X_2,\dots,X_{i-1})\right) \\
&=&P\left(\bigcup^{j+m}_{u=j+1} \{\eta_u =0\}\right) \\
&=& 1-P\left(\bigcap^{j+m}_{u=j+1} \{\eta_u=1\}\right) \\
&=&1 - a^m \ge 1-\gamma.
\end{eqnarray*}

It now follows from (\ref{eq2.9}), Lemma \ref{lm2.3}, and (\ref{eq2.4}) that
\begin{eqnarray*}
&&\lambda (\sigma(\zeta_k,k\in S),\sigma(\zeta_k,k\in T)) = \\
&&\quad =\lambda(\sigma(X_1,X_3,X_5,\dots,X_{L-1}),\sigma(X_2,X_4,X_6,\dots,X_L)) \le 3\gamma^{1/2} \le \delta.
\end{eqnarray*}
Hence by the definition of $\delta$ (just before (\ref{eq2.4}), and based on Lemma \ref{lm2.2}), (\ref{eq2.6}) holds. That completes the proof.
\end{proof}

Note that by adapting the proof of Lemma \ref{lm2.4}, one can extend Lemma \ref{lm2.4} to the broader class of random sequences in the hypothesis of Lemma \ref{lm2.3}, with the $\varepsilon \le 1/9$ there replaced by $a\in (0,1)$. However, Lemma \ref{lm2.4} in its present form will suffice for our purposes here.

This section will conclude with a lemma giving just a few related standard elementary facts which will be used later on. Here and below, for a given $a\in (0,1)$, the ``binomial distribution with parameters $0$ and $a$'' is of course the point mass at~$0$.

\begin{lemma}\label{lm2.5}
Suppose $a\in (0,1)$. Suppose $\lambda_1,\lambda_2,\lambda_3\dots$ is a sequence of positive numbers such that $\sum^\infty_{i=1} \lambda_i <\infty$. Suppose $(Y_1,Z_1),(Y_2,Z_2),(Y_3,Z_3),\dots$ is a sequence of independent random vectors such that for each $i\in \N$, (i)~the distribution of $Y_i$ is Poisson with mean $\lambda_i$, and (ii)~for each $y\in \overline\N$, the conditional distribution of $Z_i$ given $\{Y_i=y\}$ is binomial with parameters $y$ and~$a$.

(A) Then $Y:= \sum^\infty_{i=1} Y_i <\infty$ a.s., and this random variable $Y$ has the Poisson distribution with mean $\sum^\infty_{i=1} \lambda_i$.

(B) Also, $Z := \sum^\infty_{i=1} Z_i \le Y< \infty$ a.s.  Further, for any $y\in \overline\N$, the conditional distribution of $Z$ given $\{Y=y\}$ is binomial with parameters $y$ and~$a$.

(C) The ordered triplet of $\sigma$-fields $(\sigma(Y_i,i\in \N), \sigma(Y),\sigma(Z))$ is a Markov triplet.
\end{lemma}

Statement (A) holds by a simple limiting argument. Statements (B) and (C) both follow from the elementary fact that if $m$ is a nonnegative integer and $(y_1,y_2,y_3,\dots)$ is a sequence of nonnegative integers whose sum is $m$ (which allows at most finitely many $y_i$'s to be nonzero), then the event $\bigcap^\infty_{i=1} \{Y_i = y_i\}$ has positive probability and is an atom of the $\sigma$-field $\sigma(Y_1,Y_2,Y_3,\dots)$, and the conditional distribution of $Z$ given that event is binomial with parameters $m$ and~$a$.

\section{Two Markov Chains}\label{sc3}
In this section, in preparation for the main argument for Theorem \ref{th1.3} to be given in Section \ref{sc4}, the property of $\rho^*$-mixing will be verified for two classes of (nonstationary) Markov chains.

\begin{lemma}
\label{lm3.1}
Suppose $a\in (0,1)$, $p\in(0,1)$, and $N\in \N$. Suppose $Y:= (Y_0$, $Y_1$, $Y_2$, $\dots)$ is a Markov chain whose states are nonnegative integers, such that (i)~the distribution of $Y_0$ is binomial $(N,p)$, and (ii)~for each $j\in \overline\N$ and each integer $y$ such that $P(Y_j =y)>0$, the conditional distribution of $Y_{j+1}$ given $\{Y_j = y\}$ is binomial $(y,a)$.

Suppose $\varepsilon >0$, and the positive integer $m=m(a,\varepsilon)$ is as in Lemma \ref{lm2.4}. Then
\begin{equation}
\label{eq3.1}
\rho^*(Y,m) \le \varepsilon.
\end{equation}
\end{lemma}

\begin{proof}
By a standard measure-theoretic argument, the dependence coefficients
\newline$\rho^*(\cdot,n)$, $n\in\N$ for a given random sequence depend only on the distribution of that whole random sequence. Also, the distribution of a (say discrete-state) Markov chain $Y:= (Y_0$, $Y_1$, $Y_2$ ,$\dots)$ is uniquely determined by the marginal distribution of $Y_0$ and the one-step transition probabilities. Hence it suffices to carry out the proof of Lemma \ref{lm3.1} for a Markov chain $Y$ that satisfies the conditions in Lemma \ref{lm3.1} and is embedded in a convenient context.

Refer to the parameters $a$, $p$, and $N$ in the statement of Lemma \ref{lm3.1}. Let $\eta := (\eta_{h,j}$, $1\le h\le N$, $j\in\N)$ be an array of independent, identically distributed $\{0,1\}$-valued random variables such that $\rho(\eta_{1,1} =1) = a$.

Let $\zeta := (\zeta_{h,j}$, $1\le h\le N$, $j\in \overline\N)$ be an array of $\{0,1\}$-valued random variables that meets the following two conditions (interpreted appropriately if $N=1$): (i)~The random variables $\zeta_{h,0}$, $1\le h\le N$ are independent, identically distributed $\{0,1\}$-valued random variables such that $P(\zeta_{1,0}=1) = p$, with the sequence $(\zeta_{h,0}$, $1\le h\le N)$ being independent of the array $\eta$. (ii)~For each $h\in \{1,2,\dots,N\}$ and each $j\in \N$,
\begin{equation}\label{eq3.2}
\zeta_{h,j} := \zeta_{h,0} \cdot \prod^j_{i=1} \eta_{h,i}.
\end{equation}

Define the sequence $Y:= (Y_0,Y_1,Y_2,\dots)$ of (nonnegative, integer-valued) random variables as follows: For each $j\in\N$,
\begin{equation}
\label{eq3.3}
Y_j := \sum^N_{h=1} \zeta_{h,j}.
\end{equation}

By (\ref{eq3.2}), for every $h\in \{1,2,\dots,N\}$ and every $j\in \overline\N$,
\begin{equation}
\label{eq3.4}
\zeta_{h,j+1} = \zeta_{h,j} \cdot \eta_{h,j+1}.
\end{equation}
By (\ref{eq3.3}) and (\ref{eq3.4}),
\begin{equation}
\label{eq3.5}
N \ge Y_0 \ge Y_1 \ge Y_2 \ge \dots \ge 0.
\end{equation}
By (\ref{eq3.3}) and the properties of the array $\zeta$,
\begin{equation}
\label{eq3.6}
\mbox{the distribution of $Y_0$ is binomial $(N,p)$.}
\end{equation}
Our next task, starting with (\ref{eq3.6}), is to establish the distribution of the entire sequence~$Y$.

Define (with some redundancy) the $\sigma$-fields $\cG_j$, $j\in \overline\N$ as follows:
\begin{eqnarray}
\label{eq3.7}
\!\!\!&&\cG_0 := \sigma(\zeta_{h,0}, 1\le h\le N); \,\,\mbox{and} \nonumber\\
\!\!\!&&\forall\,\, j\in \N,\,\, \cG_j := \sigma(\zeta_{h,k}, 1\le h\le N, 0\le k\le j) \vee \sigma(\eta_{h,k}, 1\le h\le N, 1\le k\le j). \nonumber\\
&&
\end{eqnarray}
For each $j\in \overline\N$, the $\sigma$-field $\cG_j$ is independent of $\sigma(\eta_{h,k}$, $1\le h\le N$, $k\ge j+1)$.

Now suppose $j\in \overline\N$; and suppose $y\in \{1,2,\dots,N\}$, and $S\subset \{1,2,\dots,N\}$ is a set with cardinality $y$.  Define the event
\begin{equation}
\label{eq3.8}
A := \{\forall\, h\in S,\, \zeta_{h,j} =1;\,\,\mbox{and}\,\,\forall\, h\in \{1,\dots,N\} -S,\, \zeta_{h,j} =0\}.
\end{equation}
(If $y=N$ then $A = \bigcap^N_{h=1}\{\zeta_{h,j} =1\}$.) Suppose $G\in \cG_j$ (see (\ref{eq3.7})) is an event, and that $P(G\cap A)>0$. Then $Y_{j+1} = \sum^N_{h=1} \zeta_{h,j} \cdot \eta_{h,j+1}$ by (\ref{eq3.3}) and (\ref{eq3.4}); and hence by the sentence after (\ref{eq3.7}) and a simple argument, for every $z \in \{0,1,\dots,y\}$,
\begin{equation}
\label{eq3.9}
P(Y_{j+1} = z\mid G\cap A) = {y \choose z} a^z(1-a)^{y-z}.
\end{equation}

Next suppose again that $j\in \overline\N$ and $y\in \{1,2,\dots,N\}$. By (\ref{eq3.3}), the event $\{Y_j =y\}$ is the union of finitely many (pairwise) disjoint events of the form $A$ in (\ref{eq3.8}). Hence by (\ref{eq3.9}) and a simple calculation, if $G\in \cG_j$, $P(G\cap \{Y_j =y\}) >0$, and $z\in \{0,1,\dots,y\}$, then
\begin{equation}
\label{eq3.10}
P(Y_{j+1} = z \mid G\cap \{Y_j =y\}) = {y \choose z} a^z (1-a)^{y-z}.
\end{equation}
Of course (recall (\ref{eq3.5})) eq.\ (\ref{eq3.10}) also holds for $y=0$ (and $z=0$). Also, by (\ref{eq3.3}) and (\ref{eq3.7}), each of the random variables $Y_k$, $0\le k\le j$ is $\cG_j$-measurable. Hence (\ref{eq3.10}) has the following consequences:

The sequence $Y$ is a Markov chain. For every $j\in \overline\N$ and every $y\in \{0,1,\dots,N\}$, $P(Y_j =y)>0$ (by (\ref{eq3.6}) followed by (\ref{eq3.10}) and induction, with $G = \Omega$). Finally, for each $j\in \overline\N$ and each $y\in \{0,1,\dots,N\}$, the conditional distribution of $Y_{j+1}$ given $\{Y_j =y\}$ is binomial $(y,a)$. Hence by (\ref{eq3.6}), the sequence $Y$ meets all conditions specified in Lemma \ref{lm3.1}.

Now suppose $\varepsilon >0$, and $m=m(a,\varepsilon)$ is as in Lemma \ref{lm2.4}. To complete the proof of Lemma \ref{lm3.1}, it suffices to prove for the sequence $Y$ above that (\ref{eq3.1}) holds.

For each $h\in \{1,2,\dots,N\}$, define the random sequence $\zeta^{(h)} := (\zeta_{h,0}$, $\zeta_{h,1}$, $\zeta_{h,2}$,$\dots)$. By (\ref{eq3.2}) and the properties of the arrays $\eta$ and $\zeta$ here, for each $h\in \{1,2,\dots,N\}$, the sequence $\zeta^{(h)}$ fulfills the conditions in Lemma \ref{lm2.4}. Hence from Lemma \ref{lm2.4},
\begin{equation}
\label{eq3.11}
\forall\,\, h\in \{1,\dots,N\},\,\,\, \rho^*(\zeta^{(h)},m) \le \varepsilon.
\end{equation}
Also, by (\ref{eq3.2}) and the properties of the arrays $\eta$ and $\zeta$ here, the sequences $\zeta^{(h)}$, $h\in\{1,2,\dots,N\}$ are independent of each other. Hence by (\ref{eq3.3}), (\ref{eq3.11}), and Lemma \ref{lm2.1}, eq.\ (\ref{eq3.1}) holds. That completes the proof.
\end{proof}

\begin{lemma}\label{lm3.2}
Suppose $a\in (0,1)$ and $\lambda >0$. Suppose $Y := (Y_0,Y_1,Y_2,\dots)$ is a Markov chain with state space $\overline\N$, such that (i)~the distribution of the random variable $Y_0$ is Poisson $(\lambda)$, and (ii)~for each $j\in\overline\N$ and each $y\in\overline\N$, the conditional distribution of $Y_{j+1}$ given $\{Y_j =y\}$ is binomial $(y,a)$.

(A) For each $j\in\overline\N$, the distribution of the random variable $Y_j$ is Poisson $(\lambda a^j)$.

(B) Suppose $\varepsilon >0$, and suppose the positive integer $m=m(a,\varepsilon)$ is as in Lemma \ref{lm2.4}. Then $\rho^*(Y,m) \le \varepsilon$.
\end{lemma}

\begin{proof}
For statement (A), conditions (i) and (ii) in Lemma \ref{lm3.2} imply that $Y_1$ is Poisson $(\lambda a)$ by a standard calculation, and by repeating that argument one obtains (A) by induction.

\textit{Proof of (B).}  For each integer $n>\lambda$, let $Y^{(n)} := (Y^{(n)}_0, Y^{(n)}_1,Y^{(n)}_2,\dots)$ be a Markov chain with state space $\{0,1,\dots,n\}$ such that (i)~the distribution of $Y^{(n)}_0$ is binomial $(n,\lambda/n)$, and (ii)~for each $j\in \overline\N$ and each $y\in \{0,1,\dots,n\}$, the conditional distribution of $Y_{j+1}$ given $\{Y_j = y\}$ is binomial $(y,a)$. Then $Y^{(n)}_0$ converges in distribution to $Y_0$ (which is Poisson $(\lambda))$ as $n\to \infty$. Since the one-step transition probabilities for each of the Markov chains $Y^{(n)}$ are the same as for the Markov chain $Y$, one has that for every $j\in \overline\N$ and every choice of nonnegative integers $y_0,y_1,\dots,y_j$,
\begin{equation}
\label{eq3.12}
P\left( \bigcap^j_{i=0} \left\{Y^{(n)}_i = y_i\right\}\right) \longrightarrow P\left(\bigcap^j_{i=0} \left\{Y_i = y_i\right\}\right)\,\,\,\mbox{as $n\to \infty$.}
\end{equation}

The rest of this argument is routine, but let us go through it.  Suppose $\varepsilon >0$, and suppose $m=m(a,\varepsilon)$ is as in Lemma \ref{lm2.4}. Suppose $S$ and $T$ are nonempty, finite, disjoint subsets of $\overline\N$ such that $\dist(S,T) \ge m$. Suppose $f: \overline{\N}^I \to \R$ and $g:\overline{\N}^J \to \R$ are bounded functions, where $I$ and $J$ are the cardinalities of $S$ and $T$ respectively. To complete the proof, it suffices to show that (see the sentence right before Lemma \ref{lm2.4})
\begin{equation}
\label{eq3.13}
|\mbox{Corr}(f(Y_k, k\in S), g(Y_k,k\in T))| \le \varepsilon.
\end{equation}

Now by Lemma \ref{lm3.1}, for each integer $n>\lambda$,
\begin{equation}
\label{eq3.14}
|\mbox{Corr}(f(Y^{(n)}_k, k\in S), g(Y^{(n)}_k, k\in T))|\le \varepsilon.
\end{equation}
If the left side of (\ref{eq3.13}) is nonzero, then the left side of (\ref{eq3.14}) converges to the left side of (\ref{eq3.13}) as $n\to \infty$ by (\ref{eq3.12}) and a routine calculation.  Hence by (\ref{eq3.14}), eq.\ (\ref{eq3.13}) holds. That completes the proof.
\end{proof}

\section{Proof of Theorem \ref{th1.3}}\label{sc4}
As in the statement of Theorem \ref{th1.3}, suppose $a\in (0,1)$ and $\lambda>0$. The argument here will be divided into four ``steps''.

\textbf{Step 1.} \textit{Construction of the sequence $X$.} For each integer $\ell$ (that is, each $\ell \in \Z$), let $Y^{(\ell)} := (Y^{(\ell)}_0$, $Y^{(\ell)}_1$, $Y^{(\ell)}_2$, $\dots)$ be a Markov chain with state space $\overline\N$, such that the distribution of this Markov chain $Y^{(\ell)}$ (on $\overline{\N}^{\overline\N}$) is the same as that of the Markov chain $Y$ in Lemma \ref{lm3.2}. Let these Markov chains $Y^{(\ell)}$, $\ell\in\Z$ be constructed in such a way that they are independent of each other.

Just for convenient ``bookkeeping'' later on, for each $\ell \in\Z$ and each integer $k\le -1$, define the degenerate random variable $Y^{(\ell)}_k \equiv 0$. For each $\ell\in \Z$, thereby extend the Markov chain $Y^{(\ell)}$ (retaining that notation) to the form $Y^{(\ell)} := (Y^{(\ell)}_k, k\in \Z) = (\dots,0,0,0,Y^{(\ell)}_0,$ $Y^{(\ell)}_1,$ $Y^{(\ell)}_2, \dots)$. These random sequences $Y^{(\ell)}$, $\ell\in \Z$ are each a Markov chain, they are independent of each other, and they all have the same distribution (on, say, $\overline{\N}^\Z$). This extension does not change any of the dependence coefficients $\rho^*(Y^{(\ell)},n)$.

Now for each $\ell\in \Z$ and each $j\in \overline\N$, the distribution of the random variable $Y^{(\ell)}_j$ is Poisson with mean $\lambda a^j$ (see Lemma \ref{lm3.2}(A)). Hence in particular, for each $\ell\in \Z$, $\sum^\infty_{j=0} EY^{(\ell-j)}_j<\infty$, and hence $\sum^\infty_{j=0} Y^{(\ell-j)}_j <\infty$ a.s. Define the sequence $X := (X_k$, $k\in\Z)$ of the random variables as follows: For each $k\in\Z$,
\begin{equation}
\label{eq4.1}
X_k := \sum^\infty_{j=0} Y^{(k-j)}_j = \sum^\infty_{j=-\infty} Y^{(k-j)}_j.
\end{equation}

Since the (nonstationary) Markov chains $Y^{(\ell)}$, $\ell \in\Z$ are independent of each other and have the same distribution, it follows from an elementary (if tedious) measure-theoretic argument that this random sequence $X$ is strictly stationary. (Eq. (\ref{eq4.1}) and the resulting stationarity of $X$ are adapted from a scheme used in \cite{13} to ``convert'' a nonstationary sequence to a stationary one preserving certain properties.)

Note that by (\ref{eq4.1}) and the comments preceding it, one has (as in Lemma \ref{lm2.5}(A)) that for each $k\in\Z$,
\begin{equation}
\label{eq4.2}
\mbox{the distribution of $X_k$ is Poisson $(\lambda/(1-a))$.}
\end{equation}

\textbf{Step 2.} \textit{Verification of some features of the INAR(1) model with Poisson innovations.} For each integer $k$, referring to the comments preceding (\ref{eq4.1}), define the random variables $U_k$ and $V_k$ as follows:
\begin{equation}
\label{eq4.3}
U_k := \sum^\infty_{j=1} Y^{(k-j)}_j \,\,\mbox{and}\,\, V_k := Y^{(k)}_0.
\end{equation}
Then by (\ref{eq4.1}), for each $k\in \Z$,
\begin{equation}
\label{eq4.4}
X_k = U_k + V_k.
\end{equation}
By (\ref{eq4.3}) and the comments preceding (\ref{eq4.1}), one has that for each $k\in \Z$,
\begin{equation}
\label{eq4.5}
\mbox{the distribution of $V_k$ is Poisson $(\lambda)$.}
\end{equation}

By (\ref{eq4.1}) and (\ref{eq4.3}), for each $k\in\Z$,
\begin{equation}
\label{eq4.6}
\sigma(U_k) \subset \sigma(Y^{(\ell)},\ell\le k-1),\, \sigma(V_k) \subset \sigma(Y^{(k)}),\,\,\mbox{and}\,\,\sigma(X_k) \subset \sigma(Y^{(\ell)},\ell\le k).
\end{equation}
Since the Markov chains $Y^{(\ell)}$, $\ell\in\Z$ are independent of each other, one has by (\ref{eq4.6}) that for each $k\in\Z$,
\begin{equation}
\label{eq4.7}
\sigma(V_k)\,\,\mbox{is independent of}\,\,\sigma(U_j,V_j,X_j,\, j\le k-1)\vee \sigma(U_k).
\end{equation}
(Eqs.\ (\ref{eq4.4}), (\ref{eq4.5}), and (\ref{eq4.7}) together have the interpretation that $V_k$ is a ``Poisson innovation''.)

Next, suppose $k\in\Z$. Consider the independent random vectors
\[
\left( Y^{(k-1)}_0, Y^{(k-1)}_1\right), \left(Y^{(k-2)}_1, Y^{(k-2)}_2\right), \left(Y^{(k-3)}_2, Y^{(k-3)}_3\right),\dots\,\,.
\]
By (\ref{eq4.1}) and (\ref{eq4.3}), the first coordinates of these random vectors add up to $X_{k-1}$, and the second coordinates add up to $U_k$. From the conditions in Lemma \ref{lm3.2} (and the comments preceding (\ref{eq4.1})), the hypothesis of Lemma \ref{lm2.5} is fulfilled.

Hence by Lemma \ref{lm2.5}(B), one has that for each $k\in\Z$ and each $x\in\overline\N$,
\begin{eqnarray}
\label{eq4.8}
&&\mbox{the conditional distribution of $U_k$ given $\{X_{k-1} =x\}$}\nonumber\\
&&\mbox{is binomial with parameters $x$ and $a$.}
\end{eqnarray}
Also, from Lemma \ref{lm2.5}(C), one has that for each $k\in\Z$,
\begin{equation}
\label{eq4.9}
\left(\sigma( Y^{(k-1-j)}_j,j\ge 0), \sigma(X_{k-1}),\sigma(U_k)\right) \,\,\mbox{is a Markov triplet.}
\end{equation}

\textbf{Step 3.} \textit{Two Markov triplets.} For each $\ell\in\Z$, define the $\sigma$-field
\begin{equation}
\label{eq4.10}
\cH^{(\ell)}:= \sigma(Y^{(j)}_{\ell-j}, j\in\Z).
\end{equation}
By (\ref{eq4.1}) and (\ref{eq4.3}), for each $\ell\in\Z$,
\begin{equation}
\label{eq4.11}
\sigma(U_\ell,V_\ell,X_\ell) \subset \cH^{(\ell)}.
\end{equation}

Now for the rest of Step 3, let $k$ be an arbitrary fixed integer. For this integer $k$, the task in the rest of Step~3 here is to establish two Markov triplets connected with the conditions in Definition~\ref{df1.2}.

For each $j\in\Z$, the ordered triplet of $\sigma$-fields
\[
\left(\sigma( Y^{(j)}_u,u\le k-2-j),\sigma(Y^{(j)}_{k-1-j}),\sigma(Y^{(j)}_{k-j})\right)
\]
is a Markov triplet. Since the Markov chains $Y^{(j)}$, $j\in\Z$ are independent, one has by (\ref{eq4.10}) and a standard measure-theoretic argument that
\[
\left(\bigvee_{i\le k-2} \cH^{(i)},\cH^{(k-1)},\cH^{(k)}\right)
\]
is a Markov triplet. Hence by (\ref{eq4.11}),
\begin{equation}
\label{eq4.12}
\left( \bigvee_{i\le k-2} \cH^{(i)}, \cH^{(k-1)},\sigma(U_k)\right)
\end{equation}
is a Markov triplet.

Also, by (\ref{eq4.9}) and the fact that $Y^{(k-1-j)}_j \equiv 0$ for $j\le -1$,
\begin{equation}
\label{eq4.13}
\left(\cH^{(k-1)}, \sigma(X_{k-1}),\sigma(U_k)\right)
\end{equation}
is a Markov triplet.

Since $\sigma(X_{k-1}) \subset \cH^{(k-1)}$ by (\ref{eq4.11}), one has that for any event $C \in \sigma(U_k)$, by the sentences containing (\ref{eq4.12}) and (\ref{eq4.13}),
\[
P\left(C\,\,\bigg|\,\, \bigvee_{i\le k-1} \cH^{(i)}\right) = P\left(C\mid\cH^{(k-1)}\right) = P\left( C\mid \sigma(X_{k-1})\right)\,\,\mbox{a.s.;}
\]
and hence the ordered triplet
\[
\left(\bigvee_{i\le k-1}\cH^{(i)},\sigma(X_{k-1}),\sigma(U_k)\right)
\]
is a Markov triplet. Hence by (\ref{eq4.11}) again,
\begin{equation}
\label{eq4.14}
\left(\sigma(U_j,V_j,X_j,j\le k-1),\sigma(X_{k-1}), \sigma(U_k)\right)
\end{equation}
is a Markov triplet. Hence by (\ref{eq4.7}) and a standard measure-theoretic argument,
\[
\left(\sigma(U_j,V_j,X_j,j\le k-1),\sigma(X_{k-1}),\sigma(U_k)\vee \sigma(V_k)\right)
\]
is a Markov triplet. Hence by (\ref{eq4.4}),
\begin{equation}
\label{eq4.15}
\left(\sigma(X_j,j\le k-1),\sigma(X_{k-1}),\sigma(X_k)\right)
\end{equation}
is a Markov triplet.

Since $k\in\Z$ was arbitrary, the sequence $X$ is by (\ref{eq4.15}) a Markov chain, a property stipulated in Definition \ref{df1.2}. Eq.\ (\ref{eq4.14}) is (again for arbitrary $k\in\Z$) the other ``Markov triplet'' property stipulated in Definition \ref{df1.2}. The other properties in Definition \ref{df1.2} (and the subsequent paragraph) were verified in (\ref{eq4.2}), (\ref{eq4.4}), (\ref{eq4.5}), (\ref{eq4.7}), and (\ref{eq4.8}). That completes the verification that the sequence $X$ is an INAR(1) model with Poisson innovations.

To complete the proof of Theorem \ref{th1.3}, all that remains is to show that the sequence $X$ is $\rho^*$-mixing.

\textbf{Step 4.} \textit{Proof that $X$ is $\rho^*$-mixing.}  For each $j\in\Z$, define the ``shifted random sequence'' $\widetilde Y^{(j)} := (\widetilde Y^{(j)}_k, k\in\Z)$ by $\widetilde Y^{(j)}_k := Y^{(j)}_{k-j}$. Then by (\ref{eq4.1}) (or (\ref{eq4.10})--(\ref{eq4.11})), for each $k\in\Z$, $\sigma(X_k)\subset \bigvee_{j\in\Z} \sigma(\widetilde Y^{(j)}_k)$. Hence by Lemma \ref{lm2.1} and the first two paragraphs of Step~1, for any $n\in\N$,
\[
\rho^*(X,n) \le\sup_{j\in\Z} \rho^*(\widetilde Y^{(j)}, n) = \sup_{j\in\Z} \rho^*(Y^{(j)},n) = \rho^*(Y,n)
\]
where the sequence $Y$ is as in Lemma \ref{lm3.2}. By Lemma \ref{lm3.2}(B), that sequence $Y$ is $\rho^*$-mixing. Hence $X$ is $\rho^*$-mixing. That completes the proof of Theorem \ref{th1.3}. $\quad\square$

\acknowledgement{The author thanks Mikhail Lifshits for his encouragement and helpful comments. The author also thanks the organizers of the Workshop on Recent Developments in Statistics for Complex Dependent Data, in Loccum, Germany, August 2015. The author was inspired by talks on INAR and related processes at that conference; and part of the research for this paper was done at that conference. That workshop was partly funded by the German Academic Exchange Service (DAAD) and by the Volkswagen Foundation (program: Niedersaechsisches Vorab for female professors of Lower Saxony).}


\begin{thebibliography}{99}

\bibitem{1}  R.C.~Bradley. Every ``lower psi-mixing'' Markov chain is ``interlaced rho-mixing''.  \emph{Stochastic Process.\  Appl.} 72 (1997) 221-239.

\bibitem{2} R.C.~Bradley.
A stationary rho-mixing Markov chain which is not
``interlaced'' rho-mixing.
\emph{ J.~Theor.~Probab.\/} 14 (2001) 717-727.

\bibitem{3} R.C.~Bradley.
\emph{ Introduction to Strong Mixing Conditions\/},
Vol. 1. Kendrick Press, Heber City (Utah), 2007.

\bibitem{4} R.C.~Bradley.
On mixing properties of reversible Markov chains.  \emph{New Zealand J.\ Math.} (accepted for publication).
arXiv:1403.4895v1 [math.PR] 19 Mar 2014.

\bibitem{5} R.C.~Bradley and W.~Bryc.  Multilinear forms and measures of dependence between random varialbes.  \emph{J.\ Multivariate Anal.} 16 (1985) 335-367.

\bibitem{6} A.V.~Bulinskii.  On mixing conditions of random fields. \emph{Theor. Probab. Appl} 30 (1985) 219-220.

\bibitem{7} P. Cs\'aki and J.~Fischer. On the general notion of maximal correlation. \emph{Magyar Tud.\ Akad.\ Mat.\ Kutat\'o Int. K\"ozl.} 8 (1963) 27-51.

\bibitem{8} J.G.\ Du and Y.\ Li. The integer-valued autoregressive (INAR$(p))$ model. \emph{J.\ Time Series Anal.} 12 (1991) 129-142.

\bibitem{9} H.O.~Hirschfeld.
A connection between correlation and contingency.
\emph{ Proc. Camb. Phil. Soc.\/} 31 (1935) 520-524.

\bibitem{10} I.A.~Ibragimov.
A note on the central limit theorem for dependent
random variables.
\emph{ Theor.\ Probab.\ Appl.\/} 20 (1975) 135-141.

\bibitem{11} A.N.~Kolmogorov and Y.A. Rozanov.
On strong mixing conditions for stationary
Gaussian processes.
\emph{ Theor.~Probab.~Appl.\/} 5 (1960) 204-208.

\bibitem{12} E.\ McKenzie. Some simple models for discrete variate time series. \emph{Water Resour.\ Bull.} 21 (1985) 645-650.

\bibitem{13} R.A.\ Olshen. The coincidence of measure algebras under an exchangeable probability. \emph{Z.\ Wahrsch.\ verw.\ Gebiete} 18 (1971) 153-158.

\bibitem{14} X.\ Pedeli and D.\ Karlis. Some properties of multivariate INAR(1) processes. \emph{Comput.\ Statist.\ Data Anal.} 67 (2013) 213-225.

\bibitem{15} R.\ Peyre. Sharp equivalence between $\rho$- and $\tau$-mixing coefficients. \emph{Studia Math.} 216 (2013) 245-270.

\bibitem{16} M.~Rosenblatt. \emph{Markov Processes. Structure and Asymptotic Behavior.} Springer-Verlag, Berlin, 1971.

\bibitem{17} S.\ Schweer and C.H.\ Wei\ss.  Compound Poisson INAR(1) processes: Stochastic properties and testing for overdispersion. \emph{Comput. Statist. Data Anal.} 77 (2014) 267-284.

\bibitem{18} I.\ Silva and M.E.\ Silva. Parameter estimation for INAR processes based on high-order statistics. \emph{REVSTAT} 7 (2009) 105-117.

\bibitem{19} C.\ Stein. A bound for the error in the normal approximation to the distribution of a sum of dependent random variables. \emph{Proceedings of the Sixth Berkeley Symposium on Probability and Statistics,} Vol.\ 2, 583-602. University of California Press, Los Angeles, 1972.

\bibitem{20} H.S.\ Witsenhausen. On sequences of pairs of dependent random variables. \emph{SIAM J.\ Appl.\ Math.} 28 (1975) 100-113.
\end{thebibliography}
\end{document}